\documentclass[11pt]{amsart}

\usepackage{a4wide}
\usepackage{amssymb}
\usepackage{mathtools}
\usepackage{enumerate}
\usepackage[colorlinks=true, urlcolor=red, linkcolor=red, citecolor=blue]{hyperref}
\usepackage{xcolor}
\usepackage{esint}
\usepackage{mathrsfs}

\theoremstyle{plain}
\newtheorem{theorem}{Theorem}[section]
\newtheorem{lemma}[theorem]{Lemma}
\newtheorem{corollary}[theorem]{Corollary}

\theoremstyle{definition}
\newtheorem{definition}[theorem]{Definition}
\newtheorem{remark}[theorem]{Remark}
\newtheorem{example}[theorem]{Example}

\numberwithin{equation}{section}

\makeatletter
\@namedef{subjclassname@2020}{\textup{2020} Mathematics Subject Classification}
\makeatother

\newcommand{\R}{\mathbb{R}}
\renewcommand{\d}{\,\mathrm{d}}
\DeclareMathOperator*{\supp}{supp}
\DeclareMathOperator*{\esssup}{esssup}
\DeclareMathOperator*{\essinf}{essinf}
\DeclareMathOperator*{\wsup}{w-sup}
\DeclareMathOperator*{\winf}{w-inf}

\title{Capacitary estimates for solutions to nonlocal Dirichlet problems}

\author{Minhyun Kim}
\address{Department of Mathematics \& Research Institute for Natural Sciences, Hanyang University, 04763 Seoul, Republic of Korea}
\email{minhyun@hanyang.ac.kr}
\urladdr{https://sites.google.com/view/minhyunkim/}

\author{Se-Chan Lee}
\address{School of Mathematics, Korea Institute for Advanced Study, 02455 Seoul, Republic of Korea}
\email{sechan@kias.re.kr}
\urladdr{https://sites.google.com/view/sechanlee/}

\author{Marvin Weidner}
\address{Institute for Applied Mathematics, University of Bonn, 53115 Bonn, Germany}
\email{mweidner@ub.edu}
\urladdr{https://sites.google.com/view/marvinweidner/}

\thanks{}

\begin{document}

\begin{abstract}
We study the boundary regularity of weak solutions to nonlocal nonlinear elliptic equations with bounded measurable coefficients. Our main result establishes that a capacity density condition is equivalent to the validity of a uniform boundary H\"older estimate for solutions with H\"older continuous exterior data. More generally, we derive a fine capacitary estimate on the modulus of continuity that captures how regularity is inherited from the exterior datum to the solution.
\end{abstract}

\maketitle

\section{Introduction}

In this paper, we study the boundary regularity theory for nonlocal nonlinear operators of the form
\begin{equation}\label{eq-L}
\mathcal{L}u(x) = 2\,\mathrm{p.v.} \int_{\mathbb{R}^n} |u(x)-u(y)|^{p-2} (u(x)-u(y)) k(x, y) \,\mathrm{d}y,
\end{equation}
with a measurable kernel $k: \mathbb{R}^n \times \mathbb{R}^n \to \overline{\mathbb{R}}$ satisfying the uniform ellipticity condition
\begin{equation}\label{eq-k}
\frac{\Lambda^{-1}}{|x-y|^{n+sp}} \leq k(x, y) = k(y, x) \leq \frac{\Lambda}{|x-y|^{n+sp}},
\end{equation}
where $n \in \mathbb{N}$, $0<s<1<p<\infty$, and $\Lambda \geq 1$.

Let $\Omega \subsetneq \R^n$ be an open set. Given $x_0 \in \partial \Omega$, $R>0$, and an exterior datum $g$, we consider the localized Dirichlet problem
\begin{equation}\label{eq-loc-DP}
\left\{
\begin{aligned}
\mathcal{L}u &=0 &&\text{in }\Omega \cap B_R(x_0), \\
u &=g &&\text{in } B_R(x_0) \setminus \Omega.
\end{aligned}
\right.
\end{equation}

By the nonlocal De Giorgi--Nash--Moser theory \cite{DCKP14,DCKP16,Kas09}, it is well known that weak solutions $u$ to \eqref{eq-loc-DP} are of class $C^{\alpha}$ in the interior of $\Omega$ for some $\alpha \in (0,1)$, depending only on $n$, $s$, $p$, and $\Lambda$. The goal of this article is to investigate conditions on the domain $\Omega$ and on the exterior datum $g$ under which interior H\"older regularity persists up to the boundary. Specifically, we aim to answer the following central question:
\begin{itemize}
    \item[(i)] Is there a criterion on the domain $\Omega$ that \textit{equivalently characterizes} boundary H\"older continuity of solutions to \eqref{eq-loc-DP} with some exponent $\beta \in (0,\alpha]$?
\end{itemize}

More generally, we are interested in quantifying precisely how regularity is inherited from $g$ to $u$, even if the geometry of $\Omega$ precludes a uniform boundary H\"older estimate:
\begin{itemize}
    \item[(ii)] Can we explicitly quantify the boundary behavior of $u$ in terms of the regularity of $g$ and the geometry of $\Omega$?
\end{itemize}

For the nonlocal nonlinear operator $\mathcal{L}$ given by \eqref{eq-L}--\eqref{eq-k}, previous literature only provides sufficient criteria for (i). For instance, in \cite{KKP16,LZLH20}, it is shown that H\"older continuity at a boundary point $x_0 \in \partial \Omega$ is inherited from $g$ to $u$, provided $\Omega$ satisfies a measure density condition at $x_0$.\footnote{The \emph{measure density condition} holds at $x_0 \in \partial \Omega$ with radius $R>0$ if there exists $\eta \in (0,1)$ such that
\begin{equation}\label{eq-MDC}\tag{MDC}
|\overline{B_r(x_0)} \setminus \Omega |/|B_r(x_0)| \ge \eta \quad\text{for every }0<r\leq R.
\end{equation}} \\

However, as is well known for second-order differential operators \cite{HKM06}, rather than measuring the thickness of $\Omega^c$ near a boundary point in terms of the Lebesgue measure, a much finer quantity is given by the \emph{exterior capacitary thickness} of $\Omega$ at $x_0 \in \partial \Omega$, which is defined as
\begin{equation*}
\Theta_{x_0}(r) \coloneqq \left(\frac{\mathrm{cap}_{s, p}(\overline{B_{r}(x_0)} \setminus \Omega, B_{2r}(x_0))}{r^{n-sp}} \right)^{\frac{1}{p-1}}.
\end{equation*}

The quantity $\Theta_{x_0}(r)$ plays a central role in the recent work \cite{KLL23} (see also \cite{Bjo24,KL23,KLL25}), where a complete nonlocal analog of the celebrated Wiener criterion was established for the nonlocal nonlinear Dirichlet problem \eqref{eq-loc-DP} (with $R=\infty$) in bounded open sets $\Omega$. The main result in \cite{KLL23} states that every solution $u$ with continuous exterior datum is continuous at $x_0 \in \partial \Omega$ if and only if
\begin{align}
\label{eq-Wiener}
    \int_0^1 \Theta_{x_0}(r) \frac{\mathrm{d}r}{r} = \infty.
\end{align}

While the Wiener criterion completely characterizes boundary continuity, it does not characterize higher-order regularity as in (i). Our first main result identifies the precise geometric condition required for higher regularity, providing a complete answer to (i). We prove that boundary \emph{H\"older continuity} of $u$ at $x_0 \in \partial \Omega$ is equivalent to the capacity density condition at $x_0$.

We say that $\Omega$ satisfies the \emph{capacity density condition} at $x_0 \in \partial \Omega$ with radius $R > 0$, if there is $\theta_0 > 0$ such that
\begin{equation}
\label{eq:CDC}\tag{CDC}
\Theta_{x_0}(r) \ge \theta_0\quad\text{for every }0<r \leq R.
\end{equation}

\begin{theorem}\label{thm-CDC}
Let $x_0 \in \partial\Omega$, $\theta_0>0$, and $R>0$. The following are equivalent:
\begin{enumerate}[\textup(a\textup)]
\item\label{eq-CDC-a}
$\Omega$ satisfies \eqref{eq:CDC} at $x_0$ with constant $\theta_0$ and radius $R$;
\item\label{eq-CDC-b}
there exists $\alpha_0 = \alpha_0(n, s, p, \Lambda, \theta_0) \in (0, 1)$ such that, for every $\alpha \in (0, \alpha_0]$, there exists $C=C(n, s, p, \Lambda, \alpha, \theta_0)>0$ such that the following holds: if $g \in V^{s, p}(\Omega \cap B_R(x_0)) \cap C^\alpha(\overline{B_R(x_0)} \setminus \Omega)$ satisfies for some $L \geq 0$
\begin{equation}\label{eq-g-holder}
\omega_g(r; x_0)  \coloneqq \sup_{\overline{B_r(x_0)} \setminus \Omega}|g-g(x_0)| \leq Lr^\alpha \quad\text{for every }0<r \leq R,
\end{equation}
and if $u$ is a weak solution to \eqref{eq-loc-DP}, then
\begin{equation*}
\esssup_{\Omega\cap B_r(x_0)} |u-g(x_0)| \leq C\left(LR^{\alpha}+\mathcal{A}_{u-g(x_0)}(R; x_0)  \right) \left( \frac{r}{R} \right)^\alpha \quad\text{for every }0<r \leq R,
\end{equation*}
where
\begin{equation}\label{eq-A}
    \mathcal{A}_f(R; x_0) \coloneqq R^{-n/p} \|f\|_{L^p(\Omega \cap B_R(x_0))} + \mathrm{Tail}(f; x_0, R).
\end{equation}
\end{enumerate}
\end{theorem}

The space $V^{s,p}(\Omega \cap B_R(x_0))$ and the quantity $\mathrm{Tail}(f; x_0, R)$ are defined in Section~\ref{sec-prelim}. The notion of weak solution is defined in Definition~\ref{def-weak-DP}.

We emphasize that Theorem~\ref{thm-CDC} is a \emph{qualitative} result in the sense that it does not specify the quantitative dependence of the H\"older exponent on $\theta_0$. Moreover, it does not describe how the boundary regularity of $g$ interacts with the geometry of $\Omega$ to determine the precise modulus of continuity of $u$. Our following result establishes these quantitative relations under the assumption that $\Omega$ satisfies \eqref{eq:CDC}.

\begin{theorem}[Boundary H\"older estimate]\label{thm-holder}
Suppose that $\Omega$ satisfies \eqref{eq:CDC} at $x_0 \in \partial \Omega$ with constant $\theta_0>0$ and radius $R>0$. Let $\alpha \in (0,1)$ and let $g \in V^{s, p}(\Omega \cap B_R(x_0)) \cap C^\alpha(\overline{B_R(x_0)} \setminus \Omega)$ be such that \eqref{eq-g-holder} holds for some $L \geq 0$.
Let $u$ be a weak solution to \eqref{eq-loc-DP}. Then there exist constants $c, C>0$, depending only on $n$, $s$, $p$, and $\Lambda$, such that for every $0<r<R/4$,
    \begin{equation*}
        \esssup_{\Omega\cap B_r(x_0)}|u-g(x_0)| \leq C
        \begin{cases}
            r^{\alpha} (K_RR^{-\alpha}+L) & \text{if $\alpha<c \theta_0$},\\
            r^{\alpha} (K_RR^{-\alpha}+L \log(R/r))  & \text{if $\alpha=c\theta_0$},\\
            r^{c\theta_0}  (K_R R^{-c\theta_0}+LR^{\alpha-c\theta_0})&\text{if $\alpha>c\theta_0$},
        \end{cases}
    \end{equation*}
    where
    \begin{equation*}
       K_R\coloneqq \omega_g(R;x_0)+\mathcal A_{u-g(x_0)}(R;x_0).
    \end{equation*}
\end{theorem}

\begin{remark}
    If \eqref{eq:CDC} and \eqref{eq-g-holder} hold in a neighborhood of $x_0$ with uniform constants $\theta_0$, $\alpha$, and $L$, then by combining Theorem~\ref{thm-holder} with interior $C^{\beta}$ estimates from \cite{DCKP14}, we can obtain boundary H\"older regularity of $u$ of order $\min\{c \theta_0,\alpha,\beta\}$ in a neighborhood of $x_0$ (up to a logarithmic correction in case $\alpha = c \theta_0$). The proof is standard and goes exactly as in Step 2 of the proof of \cite[Proposition 2.6.4]{FeRo24}.
\end{remark}

H\"older regularity results for nonlocal equations in domains satisfying \eqref{eq:CDC} were previously only known for the fractional Laplacian, using an approach relying on the Caffarelli--Silvestre extension in \cite{Li25}. Hence, Theorem~\ref{thm-holder} is new both in the linear case with measurable coefficients and in the nonlinear case $p \neq 2$. Previous research has mostly focused on the optimal $C^s$ regularity of solutions in the case $p = 2$, typically in smooth domains such as $C^{1,\text{Dini}}$, $C^{1,\alpha}$, or $C^{1,1}$ domains \cite{AFLY25,AbGr23,AbRo20,BGR10,BGR14,BKK25,CKS10,CKW22,ChSo98,ChSo25b,ChSo26,GKK20,Gru15,Gru24,Gru25,KiWe24,RoSe14,RoSe16,RoSe16b,RoSe17,RoWe24}, which requires additional regularity assumptions on the kernel $k$. Moreover, only a few results seem to be available on the solution behavior in Lipschitz domains \cite{BLN24,BoNo23,Bog00,Jak02,TTV18} or for general $p \in (1,\infty)$ \cite{Bjo24,BKK25b,IaMo24,IM26,IMS16,IMS20}. We refer to \cite{BBK24,BBK25,KiLe24,LiLi17,Pal26} for additional fine qualitative results on the boundary behavior of solutions to \eqref{eq-loc-DP}, including the study of irregular boundary points, removability criteria, and equivalence of Perron and Sobolev regularity.

For local quasilinear equations, \cite{GaZi77} derived boundary H\"older estimates for weak solutions under the capacity density condition. For the Laplacian, \cite{Anc86} characterized the capacity density condition in terms of uniform decay estimates for harmonic measure. For analogous results concerning the global Dirichlet problem, see \cite{Aik02} for harmonic extensions, \cite{AS06} for $p$-harmonic extensions, and \cite{CHMPZ25} for a more general setting. More recently, for the fractional Laplacian, \cite{Li25} established a nonlocal counterpart of \cite{Aik02}.\\

We emphasize that our results also capture the boundary behavior of solutions in highly irregular regimes where \eqref{eq:CDC} fails or $g$ is not H\"older continuous. 
Specifically, the following result yields explicit quantitative estimates for solutions in terms of $\omega_g(r;x_0)$ and $\Theta_{x_0}(r)$, thereby providing a complete answer to question (ii). Moreover, it can be seen as a quantitative refinement of the Wiener criterion \eqref{eq-Wiener}.

\begin{theorem}[Wiener modulus of continuity I]\label{thm-main}
Let $x_0 \in \partial \Omega$ and $R>0$. Let $g \in V^{s, p}(\Omega \cap B_R(x_0)) \cap C(\overline{B_R(x_0)} \setminus \Omega)$ and let $u$ be a weak solution to \eqref{eq-loc-DP}. Then there exist constants $c, C>0$, depending only on $n$, $s$, $p$, and $\Lambda$, such that for every $0<r<R/4$,
\begin{align}\label{eq-main-modulus}
\begin{split}
\esssup_{\Omega\cap B_r(x_0)}|u-g(x_0)|
&\leq C\int_{4r}^R\exp\left(-c \int_{4r}^{\rho} \Theta_{x_0}(\tau)\frac{\mathrm{d} \tau}{\tau}\right) \omega_g(\rho; x_0)\Theta_{x_0}(\rho)\frac{\mathrm{d}\rho}{\rho}\\
&\quad + C\left(\omega_g(R;x_0)+\mathcal A_{u-g(x_0)}(R;x_0) \right)\exp\left(-c \int_{4r}^R \Theta_{x_0}(\rho)\frac{\mathrm{d} \rho}{\rho}\right),
\end{split}
\end{align}
where $\mathcal{A}$ is given as in \eqref{eq-A}.
\end{theorem}

Note that Theorem~\ref{thm-holder} is a consequence of Theorem~\ref{thm-main}. Furthermore, as a corollary of Theorem~\ref{thm-main}, we obtain a simpler form of the Wiener modulus of continuity.

\begin{corollary}[Wiener modulus of continuity II]\label{cor-main-2}
Under the same setting as in Theorem~\ref{thm-main}, there exist constants $c, C>0$, depending only on $n$, $s$, $p$, and $\Lambda$, such that for every $0<r<R/4$,
\begin{equation*}
        \esssup_{\Omega\cap B_r(x_0)}|u-g(x_0)|
    \leq C \omega_g(R;x_0)+ C\mathcal A_{u-g(x_0)}(R;x_0)
    \exp\left(-c \int_{4r}^R \Theta_{x_0}(\rho)\frac{\mathrm{d} \rho}{\rho}\right),
\end{equation*}
where $\mathcal{A}$ is given as in \eqref{eq-A}.
\end{corollary}

An analogous boundary Dini continuity result can be readily deduced, provided that $\omega_g$ is a Dini modulus and 
\begin{equation*}
    \Theta_{x_0}(r) \geq \kappa [\log(R/r)]^{-1} \quad \text{for all sufficiently small $r>0$},
\end{equation*}
where $\kappa$ is sufficiently large. Moreover, Corollary~\ref{cor-main-2} provides a direct quantitative proof of the sufficient part of the Wiener criterion \cite[Theorem~1.1]{KLL23}. Indeed, if $g$ is continuous at $x_0$ and the Wiener integral diverges at $x_0$, then letting $r\to 0$ and subsequently $R\to 0$ in the estimate above shows that $u$ is also continuous at $x_0$. We note that the proof of the sufficient part in \cite{KLL23} proceeds by contradiction and does not yield an explicit capacitary modulus of continuity.

Corollary~\ref{cor-main-2} establishes a nonlocal counterpart to the celebrated boundary potential theory for second-order quasi-linear elliptic problems (see for instance \cite{GaZi77} and \cite[Theorem~6.18]{HKM06}). In the nonlocal setting, such capacitary potential estimates were previously available only for the special case of the fractional Laplacian, $\mathcal{L} = (-\Delta)^s$ (see \cite{Bjo24}). The approach in \cite{Bjo24} heavily relies on the Caffarelli--Silvestre extension, a tool that is not available in our setting \eqref{eq-L}--\eqref{eq-k}. The main strategy of our proof is instead to iterate a suitable boundary capacity estimate (see Lemma~\ref{lem-cap-est}). Although this procedure is quite standard for local equations (see \cite{GaZi77,HKM06,MaZi97}), the study of the nonlocal case is more involved due to the appearance of tail terms. Managing these nonlocal contributions requires a novel iteration lemma (see Lemma~\ref{lem-iteration}).\\

Finally, we turn to a different aspect of the capacity density condition, namely its dependence on the parameters $(s, p)$. Such a dependence raises a natural question: how does the capacity density condition change when the parameters of the operators are varied? In particular, it is of interest to compare the geometric assumptions corresponding to different operators, including the local case $s=1$.

For $0<s\leq 1<p<\infty$, let $\Gamma_{s, p}$ denote the class of triples $(\Omega, x_0, R)$ such that $\Omega$ satisfies the capacity density condition \eqref{eq:CDC} at $x_0 \in \partial \Omega$ with radius $R>0$. When $s=1$, the classical $p$-condenser capacity is used in the definition of \eqref{eq:CDC}; see \cite[Chapter~2]{HKM06} for instance.

The following theorem gives a complete answer to this question by characterizing all possible inclusions between these classes.

\begin{theorem}\label{thm-CDC-comparison}
Let $0<s_i\leq 1<p_i$ be such that $(s_1, p_1) \neq (s_2, p_2)$ and $s_ip_i \leq n$ for $i=1, 2$. Then
\begin{equation}\label{eq-inclusion}
\Gamma_{s_2, p_2} \subsetneq \Gamma_{s_1, p_1}
\end{equation}
holds if and only if either
\begin{equation}\label{eq-sp-comp}
s_2p_2=s_1p_1\text{ and }p_1<p_2, \quad\text{or} \quad s_2p_2 < s_1p_1.
\end{equation}
\end{theorem}

We refer to Section \ref{sec:examples} for examples illustrating the various relations between capacity density conditions and other thickness conditions.\\

This article is structured as follows. In Section~\ref{sec-prelim}, we introduce the notion of weak solutions and the $(s,p)$-capacity. In Section~\ref{sec-3}, we analyze weak boundary conditions on arbitrary sets and record certain localized solution estimates up to the boundary. The proofs of Theorem~\ref{thm-main} and Corollary~\ref{cor-main-2} are contained in Section~\ref{sec:Wiener}, and those of Theorems~\ref{thm-CDC} and \ref{thm-holder} are given in Section~\ref{sec:CDC}. Finally, in Section~\ref{sec:examples}, we present several examples of domains and prove Theorem~\ref{thm-CDC-comparison}.

\subsection*{Acknowledgments}
Minhyun Kim was supported by the National Research Foundation of Korea (NRF) grant funded by the Korean government (MSIT) (RS-2026-25481961).
Se-Chan Lee was supported by the KIAS Individual Grant (No. MG099001) at the Korea Institute for Advanced Study.
Marvin Weidner was supported by the Deutsche Forschungsgemeinschaft (DFG,
German Research Foundation) under Germany's Excellence Strategy - EXC-2047/1 - 390685813 and through the CRC 1720 ``Analysis of criticality: from complex phenomena to models and estimates'', 53930965.

\section{Preliminaries}
\label{sec-prelim}

In this section we recall several definitions of function spaces, weak solutions, and capacities. Throughout the paper, we assume that $0<s<1<p<\infty$, $n \in \mathbb{N}$, $\Lambda \geq 1$, that $\mathcal{L}$ is an operator given by \eqref{eq-L} with a measurable kernel $k$ satisfying \eqref{eq-k}, and that $\Omega$ is a nonempty open subset of $\mathbb{R}^n$. In Section~\ref{sec:examples}, the local case $s=1$ is also considered. Moreover, the letters $c$ and $C$ denote positive constants that may change from line to line. By $A \lesssim B$ (resp.\ $A\gtrsim B$) it means that $A \leq CB$ (resp.\ $A \geq CB$) for some constant $C>0$ independent of the quantities $A$ and $B$. Also, $A \eqsim B$ means that $A \lesssim B$ and $A \gtrsim B$.

The fractional Sobolev space $W^{s, p}(\Omega)$ is the space of all functions in $L^p(\Omega)$ such that $\|u\|_{W^{s, p}(\Omega)} < \infty$, where
\begin{equation*}
\|u\|_{W^{s, p}(\Omega)}^p \coloneqq \|u\|_{L^p(\Omega)}^p + [u]_{W^{s, p}(\Omega)}^p \coloneqq \int_{\Omega} |u(x)|^p \,\mathrm{d}x + \int_{\Omega} \int_{\Omega} \frac{|u(x)-u(y)|^p}{|x-y|^{n+sp}} \,\mathrm{d}y \,\mathrm{d}x.
\end{equation*}
We denote by $W^{s, p}_{\mathrm{loc}}(\Omega)$ the space of functions that belong to $W^{s, p}(\Omega')$ for every open set $\Omega' \Subset \Omega$.

We also utilize the space
\begin{equation*}
V^{s, p}(\Omega) \coloneqq \left\{ u: \mathbb{R}^n \to \mathbb{R}: \|u\|_{V^{s, p}(\Omega)} < \infty \right\},
\end{equation*}
where
\begin{equation*}
\|u\|_{V^{s, p}(\Omega)}^p \coloneqq \|u\|_{L^p(\Omega)}^p + [u]_{V^{s, p}(\Omega)}^p \coloneqq \int_{\Omega} |u(x)|^p \,\mathrm{d}x + \int_{\Omega} \int_{\mathbb{R}^n} \frac{|u(x)-u(y)|^p}{|x-y|^{n+sp}} \,\mathrm{d}y \,\mathrm{d}x,
\end{equation*}
and
\begin{equation*}
V^{s, p}_0(\Omega) \coloneqq \overline{C_c^\infty(\Omega)}^{V^{s, p}(\Omega)}.
\end{equation*}

We recall the tail space
\begin{equation*}
L^{p-1}_{sp}(\mathbb{R}^n) \coloneqq \left\{ u \in L^{p-1}_{\mathrm{loc}}(\mathbb{R}^n): \int_{\mathbb{R}^n} \frac{|u(x)|^{p-1}}{(1+|x|)^{n+sp}} \,\mathrm{d}x < \infty \right\}.
\end{equation*}
If $u \in L^{p-1}_{sp}(\mathbb{R}^n)$, then the nonlocal tail
\begin{equation*}
\mathrm{Tail}(u; x_0, r) = \left( r^{sp} \int_{\mathbb{R}^n \setminus B_r(x_0)} \frac{|u(x)|^{p-1}}{|x-x_0|^{n+sp}} \,\mathrm{d}x \right)^{1/(p-1)}
\end{equation*}
is finite for every $x_0 \in \mathbb{R}^n$ and $r>0$.

Note that
\begin{equation*}
\mathrm{Lip}_c(\Omega) \subset V^{s, p}_0(\Omega) = \overline{C_c^\infty(\Omega)}^{W^{s, p}(\mathbb{R}^n)} \subset V^{s, p}(\Omega) \subset W^{s, p}_{\mathrm{loc}}(\Omega) \cap L^{p-1}_{sp}(\mathbb{R}^n),
\end{equation*}
and that
\begin{equation}\label{eq-V-ae}
V^{s, p}_0(\Omega) \subset \{ u \in V^{s, p}(\Omega): u=0 \text{ a.e.\ on } \mathbb{R}^n \setminus \Omega \}.
\end{equation}
The inclusion in \eqref{eq-V-ae} may be strict in general; see \cite{FSV15}. This distinction is important on rough domains. We refer the reader to Section~\ref{sec-3} for a more detailed discussion.

\begin{definition}\label{def-weak-soln}
A function $u \in W^{s, p}_{\mathrm{loc}}(\Omega) \cap L^{p-1}_{sp}(\mathbb{R}^n)$ is a {\it weak solution} (resp.\ \emph{weak supersolution}, \emph{weak subsolution}) of
\begin{equation}\label{eq-Lu=0}
\mathcal{L}u=0
\end{equation}
in $\Omega$ if
\begin{equation}\label{eq-weak}
\int_{\mathbb{R}^n} \int_{\mathbb{R}^n} |u(x)-u(y)|^{p-2}(u(x)-u(y))(\varphi(x)-\varphi(y)) k(x, y) \,\mathrm{d}y \,\mathrm{d}x \geq 0
\end{equation}
for all (resp.\ nonnegative, nonpositive) $\varphi \in C_c^{\infty}(\Omega)$.
\end{definition}

A function is a weak solution if and only if it is both a weak subsolution and a weak supersolution; see \cite[Corollary~2.3]{BBK24}. Moreover, if $u \in V^{s, p}(\Omega)$, then $u$ is a weak solution (resp.\ weak supersolution, weak subsolution) if and only if \eqref{eq-weak} holds for all (resp.\ nonnegative, nonpositive) $\varphi \in V^{s, p}_0(\Omega)$; see \cite[Proposition~2.4]{BBK24}.

We next recall the notions of capacities used in the sequel.

\begin{definition}
The \emph{condenser capacity} of a compact set $K \subset \Omega$ is defined by
\begin{equation*}
\mathrm{cap}_{s, p}(K, \Omega) = \inf_u {[u]_{W^{s, p}(\mathbb{R}^n)}^p},
\end{equation*}
where the infimum is taken over all $u \in C^\infty_c(\Omega)$ such that $u \geq 1$ on $K$.
\end{definition}

It is clear from the definition that $\mathrm{cap}_{s, p}(\cdot, \Omega)$ is monotone increasing and $\mathrm{cap}_{s, p}(K, \cdot)$ is monotone decreasing with respect to set inclusion.

\begin{definition}
The \emph{Sobolev capacity} of a compact set $K \subset \mathbb{R}^n$ is defined by
\begin{equation*}
C_{s, p}(K) = \inf_u {\|u\|_{W^{s, p}(\mathbb{R}^n)}^p},
\end{equation*}
where the infimum is taken over all $u \in C^\infty_c(\mathbb{R}^n)$ such that $u \geq 1$ on $K$.
\end{definition}

It is also clear from the definition that $C_{s, p}$ is monotone increasing with respect to set inclusion. Also, both capacities $\mathrm{cap}_{s,p}$ and $C_{s,p}$ are countably subadditive; see \cite[Proposition~3.7]{BBK25}. We also recall several properties of these capacities for the reader's convenience.

\begin{lemma}[{\cite[Lemma~2.17]{KLL23}}]\label{lem-cap}
If $0<r \leq R/2$, then
\begin{equation*}
\mathrm{cap}_{s,p}(\overline{B_r(x_0)}, B_R(x_0)) \eqsim 
\begin{cases}
r^{n-sp}  &\text{if } n>sp, \\
R^{n-sp} &\text{if } n<sp, \\
(\log(R/r))^{1-p}  &\text{if } n=sp,
\end{cases}
\end{equation*}
where the comparable constants depend only on $n$, $s$, and $p$.
\end{lemma}

\begin{lemma}[{\cite[Proposition~5.4]{BBK24}}]
Assume that $\Omega$ is bounded. Let $K \subset \Omega$ be compact. Then
\begin{equation*}
\frac{C_{s, p}(K)}{C(1+(\mathrm{diam}\,\Omega)^{sp})} \leq \mathrm{cap}_{s, p}(K, \Omega) \leq C \left( 1 + \frac{1}{\mathrm{dist}(K, \Omega^c)^p} \right) C_{s, p}(K),
\end{equation*}
where $C=C(n, s, p)>0$. In particular, $C_{s, p}(K)=0$ if and only if $\mathrm{cap}_{s, p}(K, \Omega)=0$.
\end{lemma}

\begin{lemma}[{\cite[Lemma~5.5]{BBK24}}]
Assume that $\Omega$ is bounded. Let $x_0 \in \Omega$. Then the following are equivalent:
\begin{enumerate}[\textup(a\textup)]
\item
$sp \leq n$,
\item
$C_{s, p}(\{x_0\})=0$,
\item
$\mathrm{cap}_{s, p}(\{x_0\}, \Omega)=0$.
\end{enumerate}
\end{lemma}

\section{Localized Dirichlet problems}\label{sec-3}

In this section, we introduce weak boundary conditions on arbitrary sets and use them to formulate localized Dirichlet problems. We also record the local estimates up to the boundary that will be needed in the sequel.

\subsection{Weak Dirichlet values}

The space $V^{s, p}_0(\Omega)$ is suitable for global Dirichlet problems, where the exterior Dirichlet condition is imposed on $\mathbb{R}^n \setminus \Omega$. For the localized Dirichlet problem \eqref{eq-loc-DP}, we need a corresponding notion for a portion of the exterior, such as $B_R(x_0) \setminus \Omega$.

Throughout this subsection, $D \subset \mathbb{R}^n$ denotes a nonempty bounded open set and $T \subset \mathbb{R}^n$ denotes a measurable set. In the applications below, we will typically take
\begin{equation}\label{eq-DT}
D = \Omega \cap B_R(x_0) \quad\text{and}\quad T = B_R(x_0) \setminus \Omega.
\end{equation}

We define the \emph{weak vanishing space} on $T$ relative to $D$ by
\begin{equation*}
Z^{s, p}_T(D) \coloneqq \overline{\mathrm{Lip}_c(\mathbb{R}^n \setminus T)}^{V^{s, p}(D)}.
\end{equation*}
We interpret $Z^{s, p}_T(D)$ as the space of functions that vanish on $T$ in the $V^{s,p}(D)$-sense. If $T=\mathbb{R}^n \setminus D$, then $Z^{s, p}_T(D) = V^{s, p}_0(D)$. It is straightforward to check that
\begin{equation}\label{eq-Z-ae}
Z^{s, p}_T(D) \subset \{u \in V^{s, p}(D): u=0 \text{ a.e.\ on }T\}.
\end{equation}
The inclusion in \eqref{eq-Z-ae} may be strict. In particular, when $T=\mathbb R^n\setminus D$, this reduces to the possible strict inclusion discussed in \eqref{eq-V-ae}.

\begin{definition}\label{def-weak-bdry}
Let $u, v \in V^{s, p}(D)$. We say that \emph{$u=v$ on $T$ in the $V^{s, p}(D)$-sense} if
\begin{equation*}
u-v \in Z^{s, p}_T(D).
\end{equation*}
\end{definition}

We define weak solutions to the localized Dirichlet problem \eqref{eq-loc-DP} by combining the weak formulation of the equation in Definition~\ref{def-weak-soln} with the weak interpretation of the exterior condition in Definition~\ref{def-weak-bdry}.

\begin{definition}\label{def-weak-DP}
Let $D$ and $T$ be given as in \eqref{eq-DT}, and assume $g \in V^{s, p}(D)$. We say that $u \in V^{s, p}(D)$ is a \emph{weak solution to the localized Dirichlet problem \eqref{eq-loc-DP}} if $u$ is a weak solution to $\mathcal{L}u=0$ in $D$ and $u=g$ on $T$ in the $V^{s, p}(D)$-sense.
\end{definition}

We extend Definition~\ref{def-weak-bdry} to the weak inequalities and weak extrema on $T$.

\begin{definition}
Let $u, v \in V^{s, p}(D)$ and let $T \subset \mathbb{R}^n$ be measurable.
\begin{enumerate}[(a)]
\item
We say that $u \leq v$ on $T$ in the $V^{s, p}(D)$-sense if
\begin{equation*}
(u-v)_+ \in Z^{s, p}_T(D).
\end{equation*}
\item
We define the \emph{weak supremum} and \emph{weak infimum} of $u$ on $T$ relative to $D$ by
\begin{equation*}
\wsup_{T; D} u \coloneqq \inf {\left\{ l \in \mathbb{R}: (u-l)_+ \in Z^{s, p}_T(D) \right\}} \quad\text{and}\quad \winf_{T; D} u \coloneqq - \wsup_{T; D} {(-u)},
\end{equation*}
respectively. When the underlying set $D$ is $\Omega$, we simply write $\wsup_T u = \wsup_{T; \Omega} u$ and $\winf_T u = \winf_{T; \Omega} u$. These extrema are understood as extended real numbers. 
\end{enumerate}
\end{definition}

The following lemma shows the relation between the essential supremum (resp.\ essential infimum) and the weak supremum (resp.\ weak infimum).

\begin{lemma}[Weak and essential suprema]\label{lem-wsup-esssup}
Let $u \in V^{s, p}(D)$. Then
\begin{equation}\label{eq-esssup-wsup}
\esssup_{T}{u} \leq \wsup_{T; D}{u}.
\end{equation}
The equality holds if $D$ and $T$ satisfy
\begin{equation}\label{eq-density}
Z^{s, p}_T(D) = \{u \in V^{s, p}(D): u=0 \text{ a.e.\ on }T\}.
\end{equation}
\end{lemma}

\begin{proof}
The first assertion \eqref{eq-esssup-wsup} follows from \eqref{eq-Z-ae}. Indeed, if $l \in \mathbb{R}$ is an admissible level for $\wsup_{T; D}{u}$, then $(u-l)_+ \in Z^{s, p}_T(D)$. By \eqref{eq-Z-ae}, $(u-l)_+ =0$ a.e.\ on $T$ and therefore
\begin{equation*}
\esssup_T u \leq l.
\end{equation*}
Taking the infimum over all admissible levels $l$ yields \eqref{eq-esssup-wsup}.

Now we recall that $|D|<\infty$  and assume that $D$ satisfies \eqref{eq-density}. Let $L\coloneqq \esssup_T{u}$. If $L=\infty$, then \eqref{eq-esssup-wsup} implies that $\wsup_{T; D}{u}=\infty$, and the equality follows. Suppose that $L \in \mathbb{R}$. For every $\varepsilon>0$, we have
\begin{equation*}
v \coloneqq (u-(L+\varepsilon))_+=0 \quad\text{a.e.\ on }T.
\end{equation*}
Since $|D|<\infty$, it follows from
\begin{equation*}
0 \leq v \leq |u| + |L+\varepsilon| \in L^p(D)
\end{equation*}
that $v \in L^p(D)$. Moreover, since the map $t \mapsto (t-(L+\varepsilon))_+$ is Lipschitz, we obtain
\begin{equation*}
[v]_{V^{s, p}(D)} \leq [u]_{V^{s, p}(D)} < \infty.
\end{equation*}
Thus, $v \in V^{s, p}(D)$. By \eqref{eq-density}, we have $v \in Z^{s, p}_T(D)$, meaning that
\begin{equation*}
\wsup_{T;D}{u} \leq L+\varepsilon.
\end{equation*}
Letting $\varepsilon \searrow 0$ yields the equality.

Finally, suppose that $L=-\infty$. Then $u \leq l$ a.e.\ on $T$ for every $l \in \mathbb{R}$. As above, the assumption $|D|<\infty$ guarantees that $(u-l)_+ \in V^{s, p}(D)$, and hence by \eqref{eq-density}
\begin{equation*}
(u-l)_+ \in Z^{s, p}_T(D)
\end{equation*}
for every $l \in \mathbb{R}$. Thus every real level is admissible, and therefore
\begin{equation*}
\wsup_{T; D}u=-\infty=L.
\end{equation*}
This completes the proof.
\end{proof}

\subsection{Local estimates up to the boundary}

We record the local boundedness and weak Harnack inequality up to the boundary in the form needed below.

\begin{theorem}[Local boundedness up to the boundary]\label{thm-loc-bdd}
For $x_0 \in \partial \Omega$, let $D$ and $T$ be given as in \eqref{eq-DT}, and let $0 \leq M <\infty$. Suppose that $u \in V^{s, p}(D)$ is a weak subsolution to $\mathcal{L}u=0$ in $D$ such that
\begin{equation*}
u_+ \leq M \quad\text{on }T \text{ in the $V^{s, p}(D)$-sense}.
\end{equation*}
Then there exists a constant $C=C(n, s, p, \Lambda)>0$ such that
\begin{equation*}
\esssup_{B_{R/2}(x_0)} u_M^+ \leq C \left( \fint_{B_{R}(x_0)} u_M^+(x)^p \,\mathrm{d}x \right)^{1/p} + \mathrm{Tail}(u_M^+; x_0, R/2),
\end{equation*}
where $u_M^+ \coloneqq \max\{u_+, M\}$.
\end{theorem}

\begin{theorem}[Weak Harnack inequality up to the boundary]\label{thm-WHI}
For $x_0 \in \partial \Omega$, let $D$ and $T$ be given as in \eqref{eq-DT}. Let $0\leq m<\infty$ and
\begin{equation*}
t \in
\begin{cases}
(0, \frac{n(p-1)}{n-sp}) &\text{if }sp<n, \\
(0, \infty) &\text{if }sp \geq n.
\end{cases}
\end{equation*}
Suppose that $u \in V^{s, p}(D)$ is a weak supersolution to $\mathcal{L}u=0$ in $D$ such that $u \geq 0$ a.e.\ in $B_R(x_0)$ and
\begin{equation*}
u \geq m \quad\text{on }T \text{ in the $V^{s, p}(D)$-sense}.
\end{equation*}
Then there exists a constant $C=C(n, s, p, \Lambda, t)>0$ such that
\begin{equation*}
\left( \fint_{B_{R/2}(x_0)} u_m^-(x)^t \,\mathrm{d}x \right)^{1/t} \leq C \essinf_{B_{R/4}(x_0)} u_m^- + C\,\mathrm{Tail}((u_m^-)_-; x_0, R),
\end{equation*}
where $u_m^- \coloneqq \min\{u, m\}$.
\end{theorem}

The proofs of Theorems~\ref{thm-loc-bdd} and \ref{thm-WHI} are essentially contained in \cite[Theorems~3.5 and 3.7]{KLL23} in the range $sp \leq n$. The estimates in \cite[Theorems~3.5 and 3.7]{KLL23} are stated using weak extrema. However, the final step in both arguments is a Moser iteration, which yields the essential supremum and essential infimum appearing in the statements above.

Weak extrema nevertheless play an essential role in choosing the levels $M$ and $m$. They guarantee the admissibility of test functions in the proof of Caccioppoli estimates up to the boundary. Once these boundary levels have been fixed, the Moser iterations yield the usual measure-theoretic extrema. For completeness, we make the admissibility argument explicit below. We also explain the minor modification in the fractional Sobolev embedding step which allows both estimates to cover the supercritical case $sp > n$.

We first provide a simple localization lemma, which will be used to justify the admissibility of test functions for the Caccioppoli estimates up to the boundary.

\begin{lemma}\label{lem-test-fts}
For $x_0 \in \partial \Omega$, let $D$ and $T$ be given as in \eqref{eq-DT}. Suppose that $\eta\in C_c^\infty(B_R(x_0))$ is nonnegative and that $v \in Z^{s, p}_T(D)$. Let $F: \mathbb{R} \to \mathbb{R}$ be a Lipschitz function satisfying $F(0)=0$. Then
\begin{equation*}
\eta^p F(v) \in V^{s, p}_0(D).
\end{equation*}
\end{lemma}

\begin{proof}
By assumption, there exists a sequence $\{v_i\}_{i=1}^{\infty} \subset \mathrm{Lip}_c(\mathbb{R}^n \setminus T)$ such that
\begin{equation*}
v_i \to v \quad\text{in }V^{s, p}(D).
\end{equation*}
Define
\begin{equation*}
\varphi_i \coloneqq \eta^p F(v_i).
\end{equation*}
Since $F(0)=0$, we have
\begin{equation*}
\supp{F(v_i)} \subset \supp{v_i},
\end{equation*}
and hence
\begin{equation*}
\supp{\varphi_i} \subset \supp{\eta} \cap \supp{v_i} \subset B_R(x_0) \setminus T = D \subset \Omega.
\end{equation*}
Therefore, $\varphi_i \in \mathrm{Lip}_c(D) \subset V^{s, p}_0(D)$. By the continuity of Lipschitz composition and multiplication by a fixed cutoff in $V^{s, p}(D)$, we obtain
\begin{equation*}
\eta^p F(v_i) \to \eta^p F(v) \quad\text{in }V^{s, p}(D) \quad\text{as }i \to \infty.
\end{equation*}
Since $V^{s, p}_0(D)$ is closed in $V^{s, p}(D)$, we conclude that $\eta^p F(v) \in V^{s, p}_0(D)$.
\end{proof}

\begin{proof}[Complementary proof of Theorem~\ref{thm-loc-bdd}]
The proof of \cite[Theorem~3.5]{KLL23} is based on \cite[Lemmas~3.1 and 3.4]{KLL23}. In the proof of Lemma~3.1, the admissibility of a certain test function is implicit. We first seize the opportunity to make this explicit.

Under the additional assumption that $u$ is bounded, one uses the nonnegative test function
\begin{equation*}
\varphi \coloneqq (\bar{u}^\beta - \bar{M}^\beta) \eta^p,
\end{equation*}
where $\bar{u}=u_M^+ + d$, $\bar{M}=M+d$, and $0\leq \eta \in C^\infty_c(B_R(x_0))$, with constants $\beta, d>0$. We prove that $\varphi \in V^{s, p}_0(D)$.

By the exterior condition on $T$, we have
\begin{equation*}
(u_+-M)_+ \in Z^{s, p}_T(D).
\end{equation*}
Since
\begin{equation*}
u_M^+ = (u_+-M)_+ + M,
\end{equation*}
we have
\begin{equation*}
\bar{u} = (u_+-M)_+ + \bar{M}.
\end{equation*}
Define
\begin{equation*}
F(t) \coloneqq (t+\bar{M})^\beta - \bar{M}^\beta
\end{equation*}
on the range of $(u_+-M)_+$ on $\supp{\eta}$. Since $u$ is assumed to be bounded, this range is bounded. Thus, $F$ can be extended to a globally Lipschitz function on $\mathbb{R}$ satisfying $F(0)=0$. By Lemma~\ref{lem-test-fts},
\begin{equation*}
\varphi = F((u_+-M)_+) \eta^p \in V^{s, p}_0(D).
\end{equation*}
Moreover, since $\beta>0$ and $\bar{u} \geq \bar{M}$, we have $\varphi \geq 0$. Therefore, $\varphi$ is an admissible test function.

It remains to explain why the argument also covers the supercritical case $sp>n$. The only modification is in the Sobolev embedding step in the proof of \cite[Lemma~3.4]{KLL23}. Let $\sigma \in (0,s)$ and choose $q>1$ such that
\begin{equation*}
\frac{np}{n+\sigma p} < q < \min\{p, n/\sigma\}.
\end{equation*}
Then $\sigma q<n$ and $\chi \coloneqq q^\ast_\sigma/p>1$, where
\begin{equation*}
q^\ast_\sigma = \frac{nq}{n-\sigma q}.
\end{equation*}
With this choice of $\chi$, the proof of \cite[Lemma~3.4]{KLL23} goes through exactly as in the critical case $sp=n$, by applying the fractional Sobolev inequality with $q<n/\sigma$, and then using \cite[Lemma~4.6]{Coz17} together with H\"older's inequality.
\end{proof}

The proof of Theorem~\ref{thm-WHI} follows analogously from \cite[Theorem~3.7]{KLL23}, using the preceding localization argument to justify the Caccioppoli estimates up to the boundary, replacing the weak infimum by the essential infimum arising from the Moser iteration, and applying the same modification of the fractional Sobolev embedding when $sp>n$.

\section{Quantitative Wiener modulus}
\label{sec:Wiener}

In this section, we prove Theorem~\ref{thm-main} and Corollary~\ref{cor-main-2}. More precisely, we establish boundary estimates for solutions in terms of the exterior capacitary thickness and the modulus of continuity of the Dirichlet data.

The proof of Theorem~\ref{thm-main} is based on the local boundedness up to the boundary (Theorem~\ref{thm-loc-bdd}), the boundary capacitary estimate below, and a technical iteration lemma.

\begin{lemma}[Boundary capacitary estimate]\label{lem-cap-est}
For $x_0 \in \partial \Omega$, let $D$ and $T$ be given as in \eqref{eq-DT}. Let $u$ be a weak subsolution to $\mathcal{L}u=0$ in $D$ such that $u \geq 0$ a.e.\ in $B_R(x_0)$ and
\begin{equation}\label{eq-u-vanish-T}
u=0 \quad\text{on }T \text{ in the $V^{s,p}(D)$-sense}.
\end{equation}
For $\rho>0$, set
\begin{equation*}
M(\rho) \coloneqq \esssup_{B_\rho(x_0)}{u}.
\end{equation*}
If $M(R) > 0$, then
\begin{equation*}
\Theta_{x_0}(R/4) \leq C \frac{M(R)-M(R/4)+\mathrm{Tail}((M(R)-u)_-; x_0, R)}{M(R)},
\end{equation*}
where $C=C(n, s, p, \Lambda)>0$.
\end{lemma}

\begin{proof}
The proof follows the proof of \cite[Lemma~4.1]{KLL23}, with the weak extrema in the local estimates replaced by essential extrema.

By the assumption, the function $v\coloneqq M(R)-u$ is a weak supersolution to $\mathcal{L}v=0$ in $D$ such that $0\leq v \leq M(R)$ a.e.\ in $B_{R}(x_0)$. Moreover, it follows from \eqref{eq-u-vanish-T} that
\begin{equation*}
v=M(R) \quad\text{on }T \text{ in the $V^{s,p}(D)$-sense}.
\end{equation*}
We apply Theorem~\ref{thm-WHI} to $v$ with $t=p-1$, using $m=M(R) \geq 0$ as an admissible level. Note that $v_m^-=v$ a.e.\ in $B_R(x_0)$ and $(v_m^-)_-=v_-$ in $\mathbb{R}^n \setminus B_R(x_0)$. Hence
\begin{equation*}
\left( \fint_{B_{R/2}(x_0)} v^{p-1} \,\mathrm{d}x \right)^{1/(p-1)} \leq C \essinf_{B_{R/4}(x_0)} v + C\,\mathrm{Tail}(v_-; x_0, R)
\end{equation*}
for some $C=C(n, s, p, \Lambda)>0$. Since $\essinf_{B_{R/4}(x_0)}v=M(R) - M(R/4)$, the remainder of the proof is identical to the energy estimates in \cite[Lemmas~4.1 and 4.2]{KLL23}.
\end{proof}

Next, we provide the following technical lemma, which is used in the iteration argument.

\begin{lemma}\label{lem-iteration}
Let $N \in \mathbb{N}$, $a_1>0$, $a_2>0$, $0<q<1$, and $0<\Theta_{\ast} < 1/a_1$. Suppose that a nonincreasing sequence $\{X_i\}_{i=0}^{N}$ of nonnegative real numbers and sequences $\{\theta_i\}_{i=1}^N$ and $\{h_i\}_{i=1}^N$ of nonnegative real numbers satisfy
\begin{equation*}
\theta_i\leq\Theta_\ast, \quad i=1,\dots,N,
\end{equation*}
and
\begin{equation}\label{eq-technical}
X_i \leq (1-a_1\theta_i)X_{i-1}+a_2\sum_{m=1}^{i-1}q^{i-m}(X_{m-1}-X_m) +h_i+Bq^i, \quad i=1, \dots, N,
\end{equation}
for some $B \geq 0$. Then there exist constants $c>0$ and $C>0$, depending only on $a_1$, $a_2$, $q$, and $\Theta_{\ast}$, such that for any $i=1, \dots, N$,
\begin{equation*}
X_i \leq (X_0+CB)\exp\left(-c\sum_{m=1}^{i}\theta_m\right)+C\sum_{j=1}^ih_j\exp\left(-c\sum_{m=j+1}^{i}\theta_m\right),
\end{equation*}
with the convention that $\sum_{m=i+1}^i \theta_m=0$.
\end{lemma}

\begin{proof}
Define
    \begin{equation*}
        H_0=0, \quad H_i \coloneqq \sum_{m=1}^iq^{i-m}(X_{m-1}-X_m)=X_{i-1}-X_i+qH_{i-1}, \quad i=1,\dots, N.
    \end{equation*}
Since $\{X_i\}_{i=0}^N$ is nonincreasing, we have $H_i \geq 0$ for every $i=0, \dots, N$. For $0<\delta<1$ to be chosen below, we set
    \begin{equation*}
        Y_i \coloneqq X_i+(1-\delta)H_i, \quad i=0, 1, \dots, N.
    \end{equation*}
For every $i=1, \dots, N$, using \eqref{eq-technical}, we obtain
    \begin{equation*}
    \begin{aligned}
        Y_i&=X_i+(1-\delta)(X_{i-1}-X_i)+q(1-\delta)H_{i-1}\\
        &=(1-\delta)X_{i-1}+\delta X_i+q(1-\delta)H_{i-1}\\
        &\leq (1-\delta a_1 \theta_i)X_{i-1}+\delta h_i+\delta B q^i+q(1+\delta a_2-\delta)H_{i-1}.
    \end{aligned}
    \end{equation*}

    We now choose $\delta \in (0, 1)$ sufficiently small so that
    \begin{equation*}
        q(1+\delta a_2-\delta) \leq (1-\delta a_1\Theta_{\ast})(1-\delta),
    \end{equation*}
    which is possible since $q<1$. Note that $\delta$ depends only on $a_1$, $a_2$, $q$, and $\Theta_\ast$. Since $\theta_i \leq \Theta_\ast$, this choice gives
    \begin{equation*}
        Y_i \leq (1-\delta a_1 \theta_i)Y_{i-1}+\delta h_i+\delta Bq^i.
    \end{equation*}
    Iterating this recurrence yields, for every $i=1, \dots, N$,
    \begin{equation*}
        X_i \leq Y_i \leq X_0 \prod_{m=1}^i (1-\delta a_1 \theta_m)+\delta \sum_{j=1}^i h_j \prod_{m=j+1}^i(1-\delta a_1 \theta_m)+\delta B \sum_{j=1}^i q^j \prod_{m=j+1}^i(1-\delta a_1 \theta_m),
    \end{equation*}
    with the convention that $\prod_{m=i+1}^i (1-\delta a_1 \theta_m)=1$.
    For simplicity, we set
    \begin{equation*}
        S_i \coloneqq \sum_{m=1}^i\theta_m.
    \end{equation*}
Since $1-t \leq e^{-t}$, we have
    \begin{equation*}
        \prod_{m=1}^i(1-\delta a_1 \theta_m) \leq \exp\left(-\delta a_1 S_i\right)
    \end{equation*}
    and
    \begin{equation*}
        \prod_{m=j+1}^i(1-\delta a_1 \theta_m) \leq \exp\left(-\delta a_1 (S_i-S_j)\right).
    \end{equation*}
Therefore,
    \begin{equation*}
    X_i \leq X_0 \exp\left( -\delta a_1 S_i \right) + \delta \sum_{j=1}^{i} h_j\exp\left(-\delta a_1 (S_i-S_j)\right) + \delta B \sum_{j=1}^i q^j \exp\left(-\delta a_1 (S_i-S_j)\right).
    \end{equation*}
    
    It only remains to control the term containing $B$. We claim that there exist constants $c>0$ and $C>0$, depending only on $a_1$, $a_2$, $q$, and $\Theta_\ast$, such that
    \begin{equation*}
        \sum_{j=1}^i q^j \exp\left(-\delta a_1 (S_i-S_j)\right) \leq C \exp \left(-c S_i \right).
    \end{equation*}
    Indeed, choose $j_0 \in \mathbb N \cup \{0\}$ so that 
    \begin{equation*}
       \frac{S_i}{2\Theta_{\ast}}-1 < j_0 \leq \frac{S_i}{2\Theta_{\ast}}.
    \end{equation*}
    If $j \leq j_0$, then 
    \begin{equation*}
       S_j \leq j \Theta_{\ast} \leq  j_0 \Theta_{\ast} \leq \frac{1}{2}S_i,
    \end{equation*}
    and hence
    \begin{equation*}
        \sum_{j=1}^{j_0} q^j \exp\left(-\delta a_1 (S_i-S_j)\right) \leq \exp\left(-\frac{\delta a_1}{2} S_i\right) \sum_{j=1}^{\infty}q^j=\frac{q}{1-q}\exp\left(-\frac{\delta a_1}{2}S_i\right).
    \end{equation*}
    On the other hand, we have
    \begin{equation*}
        \sum_{j=j_0+1}^i q^j \exp\left(-\delta a_1 (S_i-S_j)\right) \leq \sum_{j=j_0+1}^{\infty} q^j = \frac{q^{j_0+1}}{1-q} \leq \frac{1}{1-q} \exp\left(-\frac{|\log q|}{2\Theta_{\ast}}S_i \right).
    \end{equation*}
    A combination of these two estimates proves the claim, completing the proof.
\end{proof}

We are now ready to prove Theorem~\ref{thm-main} by using Theorem~\ref{thm-loc-bdd} and Lemmas~\ref{lem-cap-est} and \ref{lem-iteration}.

\begin{proof}[Proof of Theorem~\ref{thm-main}]
We only prove \eqref{eq-main-modulus} with $|u-g(x_0)|$ replaced by $u-g(x_0)$. Applying the same argument to $-u$ gives the full estimate \eqref{eq-main-modulus}.

Without loss of generality, we may assume that $x_0=0$ and $g(0)=0$. We set $A \coloneqq \mathcal A_u(R; 0)$ and choose the radii $r_i \coloneqq 4^{-i-1}R$ for $i=0, 1, \dots$. We also set
\begin{equation*}
        l_i\coloneqq \omega_g(4r_i; 0), \quad
        v_i\coloneqq (u-l_i)_+, \quad
        \theta_i\coloneqq \Theta_0(r_i), \quad \text{and} \quad
       X_i\coloneqq \esssup_{B_{r_i}}u_+. 
\end{equation*}
Since $g(0)=0$, we have $g \leq l_i$ pointwise in $B_{4r_i} \setminus \Omega$. It is straightforward to check that $v_i$ is a weak subsolution to \eqref{eq-Lu=0} in $\Omega \cap B_R$ such that $0 \leq v_i \leq |u|$ in $\mathbb{R}^n$ and $v_i=0$ on $B_{4r_i} \setminus \Omega$ in the $V^{s, p}(\Omega \cap B_{4r_i})$-sense. Note that $\{X_i\}_{i=0}^{\infty}$ is a nonincreasing sequence of nonnegative real numbers.

\emph{Step 1}. We first control the initial value $X_0$. By Theorem~\ref{thm-loc-bdd} applied to $v_0$, there exists a constant $C=C(n, s, p, \Lambda)>0$ such that
    \begin{equation*}
        \esssup_{B_{R/4}}v_0 \leq \esssup_{B_{R/2}}v_0 \leq C \left(\fint_{B_R}v_0^p \,\mathrm{d}x \right)^{1/p} + \mathrm{Tail}(v_0; 0, R/2).
    \end{equation*}
    For the tail term, we observe that
    \begin{align*}
            \mathrm{Tail}^{p-1}(v_0; 0, R/2)
            &=\left(\frac{R}{2}\right)^{sp} \int_{B_R \setminus B_{R/2}} \frac{|v_0(x)|^{p-1}}{|x|^{n+sp}}\,\mathrm{d}x + 2^{-sp}\mathrm{Tail}^{p-1}(v_0; 0, R)\\
            &\leq \left(\frac{R}{2}\right)^{-n} \int_{\Omega \cap B_R} |u(x)|^{p-1}\,\mathrm{d}x + 2^{-sp}\mathrm{Tail}^{p-1}(u; 0, R)\\
            &\leq CA^{p-1},
    \end{align*}
    where we used the H\"older's inequality in the last estimate. 
    Therefore, a combination of the aforementioned estimates yields that
    \begin{equation*}
        \esssup_{B_{R/4}}{(u-l_0)_+} \leq CA
    \end{equation*}
    and so
    \begin{equation*}
        X_0 \leq l_0+CA.
    \end{equation*}

\emph{Step 2}. We next derive a capacitary estimate with the scale-dependent level $l_i$. For $i=1, 2,\dots$, we set
    \begin{equation*}
        M_i(\rho)=\esssup_{B_{\rho}}v_i \quad \text{and} \quad w_i=M_i(4r_i)-v_i.
    \end{equation*}
    Since $v_i$ is a nonnegative weak subsolution to \eqref{eq-Lu=0} in $\Omega \cap B_R$ such that $v_i = 0$ in $B_{4r_i} \setminus \Omega$ in the $V^{s, p}(\Omega \cap B_{4r_i})$-sense, it follows from the capacitary estimate (Lemma~\ref{lem-cap-est}) that
    \begin{equation*}
        c \theta_iM_i(4r_i) \leq M_i(4r_i)-M_i(r_i)+\mathrm{Tail}((w_i)_-; 0, 4r_i)
    \end{equation*}
    for some $c=c(n, s, p, \Lambda)>0$. By Lemma~\ref{lem-cap}, there exists $\Theta_{\ast}=\Theta_{\ast}(n, s, p)>0$ such that $\theta_i \leq \Theta_\ast$ for all $i$. We now choose $c_0 < \min\{ c, 1/\Theta_\ast\}$ so that
    \begin{equation*}
        M_i(r_i) \leq (1-c_0\theta_i)M_i(4r_i)+\mathrm{Tail}((w_i)_-; 0, 4r_i).
    \end{equation*}
    Since
    \begin{equation*}
        M_i(r_i)=(X_i-l_i)_+ \quad \text{and} \quad M_i(4r_i)=(X_{i-1}-l_i)_+,
    \end{equation*}
    we conclude that
    \begin{equation}\label{eq-capacity-recurrence}
        X_i \leq (1-c_0\theta_i)X_{i-1}+c_0\theta_il_i+\mathrm{Tail}((w_i)_-; 0, 4r_i)
    \end{equation}
    by considering the two cases $X_i>l_i$ and $X_i\leq l_i$ separately.

\emph{Step 3}. We now estimate the tail term in \eqref{eq-capacity-recurrence}. By repeating the tail decomposition argument presented in the proof of the sufficient part of \cite[Theorem~1.1]{KLL23}, we obtain that
    \begin{equation*}
        \begin{aligned}
            \mathrm{Tail}^{p-1}((w_i)_-; 0, 4r_i) &\leq C\left(\frac{r_i}{R}\right)^{sp}\mathrm{Tail}^{p-1}(u; 0, R)+(4r_i)^{sp} \int_{B_R \setminus B_{R/4}} \frac{(w_i)^{p-1}_-(y)}{|y|^{n+sp}}\,\mathrm{d}y\\
            &\quad+C\sum_{j=1}^{i-1}4^{-spj} (M_i(4^{j+1}r_i)-M_i(4r_i))^{p-1}.
        \end{aligned}
    \end{equation*}
     The only difference is that we can exploit the inequality $(w_{i})_- \leq |u|$ without the $|l_i|$-term, since we have $l_i\geq 0$. By employing a similar argument as in Step 1, we observe that
    \begin{equation*}  
        \begin{aligned}
             \int_{B_R \setminus B_{R/4}} \frac{(w_i)^{p-1}_-(y)}{|y|^{n+sp}}\,\mathrm{d}y&\leq\int_{(B_R\setminus B_{R/4})\setminus \Omega} \frac{|g|^{p-1}(y)}{|y|^{n+sp}}\,\mathrm{d}y+\int_{(B_R\setminus B_{R/4})  \cap \Omega} \frac{|u|^{p-1}(y)}{|y|^{n+sp}}\,\mathrm{d}y\\
             &\leq CR^{-sp}\left(l_0^{p-1}+R^{-\frac{n(p-1)}{p}}\|u\|_{L^p(B_R)}^{p-1} \right).
        \end{aligned}
    \end{equation*}
Combining these estimates, we obtain
     \begin{equation}\label{eq-tail-decomp}
        \mathrm{Tail}((w_i)_-; 0, 4r_i) \leq C4^{-\frac{sp}{p-1}i}(A+l_0)+C\sum_{j=1}^{i-1}2^{-\frac{sp}{p-1}j}(M_{i}(4^{j+1}r_i)-M_{i}(4r_i)).
    \end{equation}
    Moreover, since $t \mapsto (t-l_i)_+$ is nondecreasing and $1$-Lipschitz, we observe that
    \begin{equation*}
        M_i(4^{j+1}r_i)-M_i(4r_i)=M_i(r_{i-j-1})-M_i(r_{i-1}) \leq X_{i-j-1}-X_{i-1}.
    \end{equation*}
    We then substitute \eqref{eq-tail-decomp} into \eqref{eq-capacity-recurrence} to find 
    \begin{equation*}
         X_i \leq (1-c_0\theta_i)X_{i-1}+c_0\theta_il_i+C4^{-\frac{sp}{p-1}i}(A+l_0)+C\sum_{j=1}^{i-1}2^{-\frac{sp}{p-1}j}(X_{i-j-1}-X_{i-1}).
    \end{equation*}
    By noticing that
    \begin{equation*}
        \begin{aligned}
            \sum_{j=1}^{i-1}2^{-\frac{sp}{p-1}j}(X_{i-j-1}-X_{i-1})&=\sum_{j=1}^{i-1}2^{-\frac{sp}{p-1}j}\sum^{i-1}_{m=i-j}(X_{m-1}-X_{m})\\
            &=\sum_{m=1}^{i-1}\sum^{i-1}_{j=i-m}2^{-\frac{sp}{p-1}j}(X_{m-1}-X_{m})\\
            &\leq \frac{1}{1-2^{-\frac{sp}{p-1}}}\sum_{m=1}^{i-1}2^{-\frac{sp}{p-1}(i-m)}(X_{m-1}-X_m),
        \end{aligned}
    \end{equation*}
    we conclude that
    \begin{equation*}
        \begin{aligned}
              X_i &\leq (1-c_0\theta_i)X_{i-1}+C\sum_{m=1}^{i-1}2^{-\frac{sp}{p-1}(i-m)}(X_{m-1}-X_m)+c_0\theta_il_i+C4^{-\frac{sp}{p-1}i}(A+l_0)\\
              &\leq (1-c_0\theta_i)X_{i-1}+C\sum_{m=1}^{i-1}2^{-\frac{sp}{p-1}(i-m)}(X_{m-1}-X_m)+c_0\theta_il_i+C2^{-\frac{sp}{p-1}i}(A+l_0).
        \end{aligned}
    \end{equation*}

\emph{Step 4}. We are now able to apply Lemma~\ref{lem-iteration} with $a_1=c_0$, $a_2=C$, $q=2^{-\frac{sp}{p-1}}$, $h_i=c_0\theta_il_i$ and $B=C(A+l_0)$ to deduce that 
\begin{equation*}
    \begin{aligned}
         X_i &\leq C(X_0+A+l_0)\exp\left(-c\sum_{m=1}^{i}\theta_m\right)+C\sum_{k=1}^i c_0\theta_k l_k \exp\left(-c\sum_{m=k+1}^{i}\theta_m\right)\\
    &\leq  C(l_0+A)\exp\left(-c\sum_{m=1}^{i}\theta_m\right)+C \sum_{k=1}^i \theta_k l_k \exp\left(-c\sum_{m=k+1}^{i}\theta_m\right),
    \end{aligned}
\end{equation*}
where we used the fact that $X_0 \leq l_0+CA$ from Step 1. Equivalently, we may write
\begin{equation*}
    \begin{aligned}
        \esssup_{B_{4^{-i-1}R}}u_+ &\leq C(\omega_g(R; 0)+\mathcal{A}_u(R; 0))\exp\left(-c\sum_{m=1}^{i}\theta_m\right)\\
        &\qquad+C\sum_{k=1}^i \theta_k \omega_g(4^{-k}R; 0)\exp\left(-c\sum_{m=k+1}^{i}\theta_m\right).
    \end{aligned}
\end{equation*}
By a standard argument ((1) from discrete to continuous, (2) translation back and (3) $-u$ instead of $u$), we finish the proof.
\end{proof}

Finally, we provide the proof of Corollary~\ref{cor-main-2}.

\begin{proof}[Proof of Corollary~\ref{cor-main-2}]
We first observe that
\begin{equation*}
    \frac{\mathrm{d}}{\mathrm{d}\rho}\exp\left(-c \int_{4r}^{\rho} \Theta_{x_0}(\tau)\frac{\mathrm{d} \tau}{\tau}\right)=-c\exp\left(-c \int_{4r}^{\rho} \Theta_{x_0}(\tau)\frac{\mathrm{d} \tau}{\tau}\right) \Theta_{x_0}(\rho)\frac{1}{\rho}.
\end{equation*}
Then for $r<R/4$, we have
\begin{equation*}
    \begin{aligned}
        &\int_{4r}^R\exp\left(-c \int_{4r}^{\rho} \Theta_{x_0}(\tau)\frac{\mathrm{d} \tau}{\tau}\right) \omega_g(\rho; x_0)\Theta_{x_0}(\rho)\frac{\mathrm{d}\rho}{\rho}\\
        &\quad\leq \omega_g(R; x_0) \int_{4r}^{R}\exp\left(-c \int_{4r}^{\rho} \Theta_{x_0}(\tau)\frac{\mathrm{d} \tau}{\tau}\right) \Theta_{x_0}(\rho)\frac{\mathrm{d}\rho}{\rho}\\
        &\quad\leq \frac{1}{c}\omega_g(R; x_0).
    \end{aligned}
\end{equation*}
This estimate together with \eqref{eq-main-modulus} yields Corollary~\ref{cor-main-2}.
\end{proof}

\section{Capacity density condition}
\label{sec:CDC}

In this section, we prove Theorems~\ref{thm-CDC} and \ref{thm-holder}. The proof of Theorem~\ref{thm-holder} is a direct consequence of Theorem~\ref{thm-main}.

\begin{proof}[Proof of Theorem~\ref{thm-holder}]
    Applying Theorem~\ref{thm-main}, we obtain
    \begin{equation*}   
    \begin{aligned}
        \esssup_{\Omega\cap B_r(x_0)}|u-g(x_0)| &\leq CK_R\exp\left(-c \int_{4r}^R \Theta_{x_0}(\tau)\frac{\mathrm{d} \tau}{\tau}\right)\\
        &\qquad+C\int_{4r}^R\exp\left(-c \int_{4r}^{\rho} \Theta_{x_0}(\tau)\frac{\mathrm{d} \tau}{\tau}\right) \omega_g(\rho; x_0)\Theta_{x_0}(\rho)\frac{\mathrm{d}\rho}{\rho}.
    \end{aligned}
\end{equation*}
Due to the capacity density condition, we have 
\begin{equation*}
    \exp\left(-c \int_{4r}^{\rho} \Theta_{x_0}(\tau)\frac{\mathrm{d} \tau}{\tau}\right) \leq \left(\frac{4r}{\rho}\right)^{c\theta_0}.
\end{equation*}
Since $\Theta_{x_0} \leq C(n, s, p)$, we obtain
\begin{equation*}
     \begin{aligned}
        \esssup_{\Omega\cap B_r(x_0)}|u-g(x_0)| \leq CK_R\left(\frac{r}{R}\right)^{c\theta_0}+C\int_{4r}^R\left(\frac{r}{\rho}\right)^{c\theta_0} L\rho^{\alpha}\frac{\mathrm{d}\rho}{\rho}.
    \end{aligned}
\end{equation*}
The desired estimate follows from dividing cases depending on $\alpha$ and $c\theta_0$.
\end{proof}

We next prove Theorem~\ref{thm-CDC}. A main tool for the proof of the implication (\ref{eq-CDC-b})$\Rightarrow$(\ref{eq-CDC-a}) is the $\mathcal{L}$-potential; let $K \subset \Omega$ be compact and let $\psi \in C^\infty_c(\Omega)$ be such that $\psi=1$ on $K$. The unique weak solution $u \in C(\Omega \setminus K)$ to $\mathcal{L}u=0$ in $\Omega \setminus K$ such that $u-\psi \in V^{s, p}_0(\Omega\setminus K)$ is called the \emph{$\mathcal{L}$-potential} of $K$ in $\Omega$ and denoted by $\mathfrak{R}(K, \Omega)$.

\begin{proof}[Proof of Theorem~\ref{thm-CDC}]
The implication (\ref{eq-CDC-a})$\Rightarrow$(\ref{eq-CDC-b}) follows from Theorem~\ref{thm-holder}.

Let us now prove (\ref{eq-CDC-b})$\Rightarrow$(\ref{eq-CDC-a}).
We may assume that $x_0=0$. Fix $0<r<R$. Choose a radially non-increasing function $\eta \in C^\infty_c(B_1)$ such that $0 \leq \eta \leq 1$ in $\mathbb{R}^n$ and $\eta=1$ on $\overline{B_{1/2}}$. Let $\psi_r(x) \coloneqq \eta(x/r)$. Then $0 \leq \psi_r \leq \psi_{2r} \leq 1$, $\psi_r=1$ on $\overline{B_{r/2}}$, and $\psi_{2r}=1$ on $\overline{B_r}$.

Let $u$ be a weak solution to
\begin{equation*}
\left\{
\begin{aligned}
\mathcal{L}u &=0 &&\text{in }\Omega \cap B_r, \\
u &=\psi_r &&\text{in }\mathbb{R}^n \setminus (\Omega \cap B_r).
\end{aligned}
\right.
\end{equation*}
By the comparison principle, $0\leq u \leq 1$ a.e.\ in $\mathbb{R}^n$. Note that $u$ is a weak solution to \eqref{eq-loc-DP} with $g$ and $R$ replaced by $\psi_r$ and $r$, respectively. It thus follows from the assumption~\eqref{eq-CDC-b} that
\begin{equation*}
\esssup_{\Omega \cap B_{\varepsilon r}} |u(x)-\psi_r(0)| \leq C \varepsilon^\alpha \leq 1/2,
\end{equation*}
provided that $\varepsilon \in (0, 1/4)$ is sufficiently small. This in particular implies that
\begin{equation}\label{eq-u-large}
u \geq 1/2 \quad\text{a.e.\ in }\Omega \cap B_{\varepsilon r}.
\end{equation}

We now set $K\coloneqq \overline{B_r} \setminus \Omega$ and denote by $v=\mathfrak{R}(K, B_{2r})$ the $\mathcal{L}$-potential of $K$ in $B_{2r}$. We claim that
\begin{equation}\label{eq-CP}
u \leq v \quad\text{a.e.\ in } \Omega \cap B_r.
\end{equation}
By the comparison principle, it suffices to prove that
\begin{equation}\label{eq-u-v-test}
(u-v)_+ \in V^{s, p}_0(\Omega \cap B_r).
\end{equation}
Indeed, since $u-\psi_r \in V^{s, p}_0(\Omega \cap B_r) \subset V^{s, p}_0(B_r)$ and $\psi_r \in C^\infty_c(B_r) \subset V^{s, p}_0(B_r)$, we have $u \in V^{s, p}_0(B_r)$. Moreover, since $u, v \geq 0$ a.e.\ in $\mathbb{R}^n$, we have $0 \leq (u-v)_+ \leq u$ a.e.\ in $\mathbb{R}^n$. It follows from \cite[Corollary~2.9]{BBK24} that
\begin{equation*}
(u-v)_+ \in V^{s, p}_0(B_r).
\end{equation*}
On the other hand, since
\begin{align*}
u-\psi_r &\in V^{s, p}_0(\Omega \cap B_r) \subset V^{s, p}_0(\mathbb{R}^n \setminus K) \quad\text{and} \\
v-\psi_{2r} &\in V^{s, p}_0(B_{2r} \setminus K) \subset V^{s, p}_0(\mathbb{R}^n \setminus K),
\end{align*}
we obtain $(u-\psi_r)_+, (\psi_{2r}-v)_+ \in V^{s, p}_0(\mathbb{R}^n \setminus K)$. Since $\psi_r \leq \psi_{2r}$, we have
\begin{equation*}
0 \leq (u-v)_+ \leq (u-\psi_r)_+ + (\psi_r - \psi_{2r})_+ + (\psi_{2r} - v)_+ = (u-\psi_r)_+ + (\psi_{2r} - v)_+.
\end{equation*}
Therefore, again by \cite[Corollary~2.9]{BBK24},
\begin{equation*}
(u-v)_+ \in V^{s, p}_0(\mathbb{R}^n \setminus K).
\end{equation*}
An application of \cite[Lemma~2.8]{BBK24} proves \eqref{eq-u-v-test}, and the claim \eqref{eq-CP} follows.

Now, it follows from \eqref{eq-u-large} and \eqref{eq-CP} that $v \geq u \geq 1/2$ a.e.\ in $\Omega \cap B_{\varepsilon r}$. Also, note that $v = 1$ a.e.\ on $B_{\varepsilon r} \setminus \Omega$. Thus, $v \geq 1/2$ a.e.\ in $B_{\varepsilon r}$. We define the function
\begin{equation*}
w \coloneqq \min\{2v, 1\}.
\end{equation*}
After redefining $w$ on a set of measure zero if necessary, we have $0 \leq w \leq 1$ in $\mathbb{R}^n$ and $w=1$ in a neighborhood of $\overline{B_{\varepsilon r/2}}$. Thus, it follows from \cite[Proposition~6.2]{BBK24} that $w$ is admissible for the condenser capacity $\mathrm{cap}_{s, p}(\overline{B_{\varepsilon r/2}}, B_{2r})$. Therefore,
\begin{equation*}
\mathrm{cap}_{s, p}(\overline{B_{\varepsilon r/2}}, B_{2r}) \leq [w]^p_{W^{s, p}(\mathbb{R}^n)} \leq 2^p [v]_{W^{s, p}(\mathbb{R}^n)}^p.
\end{equation*}
By the minimizing property of the $\mathcal{L}$-potential (see \cite[Lemma~2.16(iii)]{KLL23}) and the ellipticity \eqref{eq-k} of the kernel,
\begin{equation*}
[v]_{W^{s, p}(\mathbb{R}^n)}^p \leq \Lambda \int_{\mathbb{R}^n} \int_{\mathbb{R}^n} |v(x)-v(y)|^p k(x, y) \,\mathrm{d}y \,\mathrm{d}x \leq \Lambda^2 \mathrm{cap}_{s, p}(K, B_{2r}).
\end{equation*}
On the other hand, since $\mathrm{cap}_{s, p}(\overline{B_{\varepsilon r/2}}, B_{2r}) \eqsim \mathrm{cap}_{s, p}(\overline{B_{r}}, B_{2r})$ by Lemma~\ref{lem-cap}, we obtain
\begin{equation*}
\mathrm{cap}_{s, p}(\overline{B_{r}}, B_{2r}) \leq C \mathrm{cap}_{s, p}(K, B_{2r})
\end{equation*}
for some $C=C(n, s, p, \varepsilon)>0$. Therefore, (\ref{eq-CDC-a}) holds.
\end{proof}

\section{Relations among the geometric conditions}
\label{sec:examples}

In this section, we illustrate the relations among the geometric conditions considered in this paper and prove Theorem~\ref{thm-CDC-comparison}.

For $0<s<1<p<\infty$, the following implications hold:
\begin{equation}\label{eq-hierarchy}
\eqref{eq-MDC} \implies \eqref{eq:CDC} \implies \eqref{eq-Wiener}.
\end{equation}
The first implication follows from the fractional Poincar\'e inequality\footnote{Suppose that \eqref{eq-MDC} holds. If $\varphi \in C^\infty_c(B_{2r}(x_0))$ is any function such that $\varphi \geq 1$ on $K \coloneqq \overline{B_r(x_0)} \setminus \Omega$, then by the fractional Poincar\'e inequality and Lemma~\ref{lem-cap}
\begin{equation*}
[\varphi]_{W^{s, p}(\mathbb{R}^n)}^p \geq c r^{-sp} \int_{B_{2r}(x_0)} |\varphi|^p \d x \geq cr^{-sp} |K| \geq c\eta r^{n-sp} \geq c(n, s, p, \eta) \mathrm{cap}_{s, p}(\overline{B_r(x_0)}, B_{2r}(x_0)).
\end{equation*}
Taking the infimum over all such $\varphi$ yields \eqref{eq:CDC}.}, whereas the second one follows immediately from the definition. Examples~\ref{ex-CDC-without-MDC} and \ref{ex-Wiener-without-CDC} show that both implications are strict.

\begin{example}[\eqref{eq:CDC} without \eqref{eq-MDC}]\label{ex-CDC-without-MDC}
Assume first that $sp < n$. Let $K \subset B_{1/2}$ be a compact Ahlfors $d$-regular set for some $n-sp<d<n$, i.e., there exists $C \geq 1$ such that
\begin{equation*}
C^{-1} r^d \leq \mathcal{H}^d(K \cap B_r(x)) \leq C r^d    
\end{equation*}
for every $x\in K$ and every $0<r<\mathrm{diam}\,K$. Let $\Omega \coloneqq B_1 \setminus K$. Since $d<n$, we have $|K|=0$, and hence
\begin{equation*}
|B_r(x_0) \setminus \Omega| = |K \cap B_r(x_0)| =0
\end{equation*}
for every $x_0 \in K$ and all sufficiently small $r>0$. Thus the measure density condition \eqref{eq-MDC} fails at every $x_0 \in K$.

On the other hand, by the trace theorem for $d$-sets \cite[Chapter~VII, Theorem~1]{JW84}, followed by the Sobolev--Lorentz embedding on Ahlfors regular sets \cite{Dyd10}, we have
\begin{equation*}
[\varphi]_{W^{s, p}(\mathbb{R}^n)}^p \geq c \left( \int_{K \cap B_r(x_0)} |\varphi|^q \,\mathrm{d}\mathcal{H}^d \right)^{p/q}
\end{equation*}
for every $\varphi \in C^\infty_c(B_{2r}(x_0))$, where $q=pd/(n-sp)$. If $\varphi \geq 1$ on $K \cap B_r(x_0)$, then the lower Ahlfors regularity gives
\begin{equation*}
[\varphi]_{W^{s, p}(\mathbb{R}^n)}^p \geq c \left( \mathcal{H}^d(K \cap B_r(x_0)) \right)^{p/q} \geq c r^{dp/q} = c r^{n-sp}.
\end{equation*}
Taking the infimum over all such $\varphi$ and using Lemma~\ref{lem-cap} yields
\begin{equation*}
\mathrm{cap}_{s, p}(\overline{B_r(x_0)} \setminus \Omega, B_{2r}(x_0)) = \mathrm{cap}_{s, p}(K \cap \overline{B_r(x_0)}, B_{2r}(x_0)) \geq c \, \mathrm{cap}_{s, p}(\overline{B_r(x_0)}, B_{2r}(x_0))
\end{equation*}
for sufficiently small $r$. Thus, $\Omega$ satisfies the \eqref{eq:CDC} at every $x_0 \in K$.

If $sp=n$, then we choose $\sigma \in (0,s)$ sufficiently close to $s$ so that $n-\sigma p<d$, and set $q_\sigma \coloneqq pd/(n-\sigma p)$. Applying the preceding subcritical trace estimate with $s$ replaced by $\sigma$, and then using
\begin{equation*}
[\varphi]_{W^{\sigma, p}(\mathbb{R}^n)}^p \leq C r^{(s-\sigma)p} [\varphi]_{W^{s, p}(\mathbb{R}^n)}^p \quad\text{for }\varphi \in C^{\infty}_c(B_{2r}(x_0)),
\end{equation*}
we obtain
\begin{equation*}
r^{n-\sigma p} \lesssim \left( \mathcal{H}^d(K \cap B_r(x_0)) \right)^{p/q_\sigma} \lesssim r^{(s-\sigma)p} [\varphi]_{W^{s, p}(\mathbb{R}^n)}^p.
\end{equation*}
As in the subcritical case, this yields the capacity density condition \eqref{eq:CDC} at every $x_0 \in K$.

Finally, if $sp>n$, then \eqref{eq:CDC} holds automatically at every boundary point. Indeed, for $x_0 \in \partial \Omega$,
\begin{equation*}
\mathrm{cap}_{s, p}(\overline{B_r(x_0)} \setminus \Omega, B_{2r}(x_0)) \geq \mathrm{cap}_{s, p}(\{x_0\}, B_{2r}(x_0)) \eqsim r^{n-sp}.
\end{equation*}
Thus the punctured ball $\Omega = B_1 \setminus \{0\}$ satisfies \eqref{eq:CDC} at the origin, whereas \eqref{eq-MDC} fails there since $|\overline{B_r} \setminus \Omega|=0$ for $0<r<1$.
\end{example}

\begin{example}[Wiener regularity without \eqref{eq:CDC}]\label{ex-Wiener-without-CDC}
We construct a domain with $0 \in \partial \Omega$ and $\Theta_0(r_i) \eqsim 1/i$ for all $i$ by choosing suitable radii $r_i$ and placing a small ball in each dyadic annulus; hence the \eqref{eq:CDC} fails, whereas the Wiener integral still diverges.

Assume first that $sp<n$. For $i=1, 2, \dots$, let
\begin{equation}\label{eq-radii}
r_i \coloneqq 2^{-i}, \quad x_i \coloneqq \frac{3}{4}r_i e_1, \quad \rho_i \coloneqq \frac{r_i}{16} i^{-\frac{p-1}{n-sp}},
\end{equation}
and set
\begin{equation*}
K_i \coloneqq \overline{B_{\rho_i}(x_i)}, \quad K \coloneqq \{0\} \cup \bigcup_{i=1}^\infty K_i, \quad \Omega \coloneqq B_1 \setminus K.
\end{equation*}
Then $K$ is compact, $0 \in \partial \Omega$, and $K_i \subset B_{r_i} \setminus \overline{B_{r_{i+1}}}$.

By Lemma~\ref{lem-cap}, the monotonicity and countable subadditivity of the condenser capacity (see \cite[Proposition~3.7]{BBK25}), and the fact that points have zero condenser capacity when $sp \leq n$ (see \cite[Lemma~5.5]{BBK24}), we have
\begin{equation*}
\mathrm{cap}_{s, p}(K \cap \overline{B_{r_i}}, B_{2r_i}) \eqsim \rho_i^{n-sp} \eqsim r_i^{n-sp} i^{-(p-1)}.
\end{equation*}
Indeed, the lower bound follows from $K_i \subset K \cap \overline{B_{r_i}}$, while the upper bound follows from
\begin{equation*}
\mathrm{cap}_{s, p}(K \cap \overline{B_{r_i}}, B_{2r_i}) \leq \sum_{j=i}^\infty \mathrm{cap}_{s, p}(K_j, B_{2r_i}) \lesssim \sum_{j=i}^\infty \rho_j^{n-sp} \lesssim \rho_i^{n-sp}.
\end{equation*}
Consequently, $\Theta_0(r_i) \eqsim 1/i \to 0$ as $i \to \infty$, and hence $\Omega$ does not satisfy the \eqref{eq:CDC} at the origin.

On the other hand, if $\frac{13}{16}r_i \leq r \leq r_i$, then $K_i \subset K \cap (\overline{B}_r \setminus B_{r/2})$, and hence $\Theta_0(r) \gtrsim 1/i$. Therefore,
\begin{equation*}
\int_0^{1/2} \Theta_0(r) \frac{\mathrm{d}r}{r} \gtrsim \sum_{i=1}^\infty \frac{1}{i} \int_{13r_i/16}^{r_i} \frac{\mathrm{d}r}{r} = \log \left( \frac{16}{13} \right) \sum_{i=1}^\infty \frac{1}{i} = \infty,
\end{equation*}
and so the Wiener condition \eqref{eq-Wiener} holds at the origin.

Next, we assume that $sp=n$. In this case, we let
\begin{equation*}
r_i \coloneqq \exp(-2^{i}) \quad x_i \coloneqq \frac{3}{4} r_i e_1, \quad \rho_i \coloneqq \frac{r_i}{16} \exp(-i),
\end{equation*}
instead of \eqref{eq-radii}. Then
\begin{equation*}
\mathrm{cap}_{s, p}(K \cap \overline{B_{r_i}}, B_{2r_i}) \geq \mathrm{cap}_{s, p}(K_i, B_{2r_i}) \eqsim \left( \log \frac{r_i}{\rho_i} \right)^{1-p} \eqsim i^{1-p}
\end{equation*}
and
\begin{align*}
\mathrm{cap}_{s, p}(K \cap \overline{B_{r_i}}, B_{2r_i})
&\leq \sum_{j=i}^\infty \mathrm{cap}_{s, p}(K_j, B_{2r_i}) \lesssim \sum_{j=i}^\infty \left( 2^j-2^i+j+\log 16 \right)^{1-p} \\
&\lesssim (i+\log 16)^{1-p} + \sum_{j=i+1}^\infty 2^{(1-p)j} \lesssim i^{1-p}.
\end{align*}
Thus, the Wiener condition \eqref{eq-Wiener} holds at the origin, even though the \eqref{eq:CDC} fails there.
\end{example}

The following example illustrates that even the weakest condition in the hierarchy \eqref{eq-hierarchy} may fail.

\begin{example}[Failure of the Wiener condition \eqref{eq-Wiener}]\label{ex-failure-Wiener}
As mentioned in Example~\ref{ex-CDC-without-MDC}, if $sp>n$, then \eqref{eq:CDC} holds and hence so does \eqref{eq-Wiener} at every boundary point.

When $sp \leq n$, boundary points with finite Wiener integral do occur. For example, the origin is an irregular boundary point of the punctured ball $B_1 \setminus \{0\}$, and its Wiener integral is zero. On the other hand, the capacitary spine constructed in \cite[Example~5.11]{BBK25} has the origin as an irregular boundary point with finite Wiener integral. (In the first example, the origin is semiregular, whereas in the second example it is strongly irregular; see \cite{BBK24} for the definitions.)
\end{example}

So far, we have considered geometric properties of the capacity density condition for fixed $(s,p)$. We next study how this condition depends on the parameters of the underlying operator and prove Theorem~\ref{thm-CDC-comparison}. Before the proof, we provide two auxiliary lemmas. The first one gives a characterization of the Sobolev capacity in terms of the Wolff potential.

\begin{lemma}\label{lem-dual-cap}
Let $0<s\leq 1<p$ and assume $sp \leq n$. If $K \subset \mathbb{R}^n$ is compact, then
\begin{equation*}
C_{s, p}(K) \eqsim \sup \left\{ \mu(K): \mu \in \mathcal{M}^+(K), {\bf W}^\mu_{s, p} \leq 1 \text{ on }\supp{\mu} \right\},
\end{equation*}
where
\begin{equation*}
{\bf W}^\mu_{s, p}(x) \coloneqq \int_0^1 \left( \frac{\mu(B_r(x))}{r^{n-sp}} \right)^{1/(p-1)} \frac{\mathrm{d}r}{r}
\end{equation*}
is the Wolff potential of $\mu$.
\end{lemma}

\begin{proof}
Let $\mathcal{C}_{s, p}$ denote the Bessel capacity corresponding to $s$ and $p$; see \cite[Definition~2.2.6]{AH96} for the definition of $\mathcal{C}_{s, p}$ (which is denoted by $C_{s, p}$ in \cite{AH96}). In \cite[Section~4]{HW83}, the authors introduce a lower semicontinuous modified Wolff potential $\mathscr{W}^\mu_{s, p}$, which satisfies
\begin{equation}\label{eq-W-comp}
\mathscr{W}^\mu_{s, p} \eqsim {\bf W}^\mu_{s, p};
\end{equation}
see \cite[p.~175]{HW83}. They also introduce the corresponding Wolff capacity $\mathscr{C}_{s, p}$ and show that
\begin{equation}\label{eq-cap-comp}
\mathscr{C}_{s, p} \eqsim \mathcal{C}_{s, p};
\end{equation}
see \cite[p.~176]{HW83}. Moreover, \cite[Proposition~5]{HW83} gives
\begin{equation*}
\mathscr{C}_{s, p} = \sup \left\{ \mu(K): \mu \in \mathcal{M}^+(K), \mathscr{W}^\mu_{s, p} \leq 1 \text{ on }\supp{\mu} \right\}.
\end{equation*}
Using \eqref{eq-W-comp}, \eqref{eq-cap-comp}, and the homogeneity of two Wolff potentials with respect to multiplication of the measure, we conclude that
\begin{equation*}
\mathcal{C}_{s, p} \eqsim \sup \left\{ \mu(K): \mu \in \mathcal{M}^+(K), {\bf W}^\mu_{s, p} \leq 1 \text{ on }\supp{\mu} \right\}.
\end{equation*}
Finally, the desired result follows from the comparability of the Bessel capacity $\mathcal{C}_{s, p}$ to the Sobolev capacity $C_{s, p}$ (see \cite[Remark~8.2]{BBK25}).
\end{proof}

\begin{lemma}\label{lem-cap-sp-comp}
Let $0<s_i\leq 1<p_i$ be such that $(s_1, p_1) \neq (s_2, p_2)$ and $s_ip_i \leq n$, $i=1, 2$. Suppose that \eqref{eq-sp-comp} holds. Then there exists $c>0$, depending only on $s_1$, $p_1$, $s_2$, $p_2$, and $n$, such that
\begin{equation}\label{eq-sp1-sp2}
C_{s_1, p_1}(K) \geq c\,C_{s_2, p_2}(K)
\end{equation}
for every compact set $K \subset \overline{B_1}$.
\end{lemma}

\begin{proof}
By Lemma~\ref{lem-dual-cap}, there exists $\mu \in \mathcal{M}^+(K)$ such that
\begin{equation*}
\mu(K) \eqsim C_{s_2, p_2}(K) \quad\text{and}\quad {\bf W}^{\mu}_{s_2, p_2} \leq 1 \text{ on }\supp{\mu}.
\end{equation*}
We first observe that there exists a constant $C=C(n, s_2, p_2)>0$ such that
\begin{equation*}
\frac{\mu(B_r(x))}{r^{n-s_2p_2}} \leq C
\end{equation*}
for every $x \in \supp{\mu}$ and $0<r\leq 1$. Indeed, if $0<r\leq 1/2$, then
\begin{equation*}
\left( \frac{\mu(B_r(x))}{r^{n-s_2p_2}}\right)^{1/(p_2-1)}  \leq \int_r^{2r} \left( \frac{\mu(B_\rho(x))}{(\rho/2)^{n-s_2p_2}} \right)^{1/(p_2-1)} \frac{\mathrm{d}\rho}{\rho/2} \leq C{\bf W}^\mu_{s_2, p_2}(x) \leq C,
\end{equation*}
and if $1/2<r\leq 1$, then
\begin{equation*}
\frac{\mu(B_r(x))}{r^{n-s_2p_2}} \leq 2^{s_2p_2-n} \mu(K) \lesssim C_{s_2, p_2}(K) \leq C_{s_2, p_2}(\overline{B_1}) \leq C.
\end{equation*}

Next, we estimate ${\bf W}^\mu_{s_1, p_1}$ from above on $\supp{\mu}$. If $s_2p_2=s_1p_1$ and $p_1<p_2$, then it follows from
\begin{equation*}
\left( \frac{\mu(B_r(x))}{r^{n-s_1p_1}} \right)^{1/(p_1-1)} = \left( \frac{\mu(B_r(x))}{r^{n-s_2p_2}} \right)^{1/(p_1-1)} \leq C \left( \frac{\mu(B_r(x))}{r^{n-s_2p_2}} \right)^{1/(p_2-1)}
\end{equation*}
that
\begin{equation*}
{\bf W}^\mu_{s_1, p_1} \leq C {\bf W}^\mu_{s_2, p_2} \leq C \quad\text{on }\supp{\mu}.
\end{equation*}
If $s_2p_2<s_1p_1$, then we compute
\begin{equation*}
\left( \frac{\mu(B_r(x))}{r^{n-s_1p_1}} \right)^{1/(p_1-1)} = \left( \frac{\mu(B_r(x))}{r^{n-s_2p_2}} \right)^{1/(p_1-1)} r^{(s_1p_1 - s_2p_2)/(p_1-1)},
\end{equation*}
and deduce that
\begin{equation*}
{\bf W}^\mu_{s_1, p_1}(x) \leq C \int_0^1 r^{(s_1p_1-s_2p_2)/(p_1-1)} \frac{\mathrm{d}r}{r} \leq C.
\end{equation*}
In either case, we have
\begin{equation*}
{\bf W}^\mu_{s_1, p_1} \leq C_0 \quad\text{on }\supp{\mu}
\end{equation*}
for some $C_0=C_0(n, s_1, p_1, s_2, p_2)>0$.

Define $\tilde{\mu} \coloneqq C_0^{1-p_1}\mu$. Then
\begin{equation*}
\tilde{\mu}(K) = C_0^{1-p_1} \mu(K) \gtrsim C_{s_2, p_2}(K) \quad\text{and}\quad {\bf W}^{\tilde{\mu}}_{s_1, p_1} \leq 1 \text{ on }\supp{\mu}.
\end{equation*}
Therefore, \eqref{eq-sp1-sp2} follows from Lemma~\ref{lem-dual-cap}.
\end{proof}

\begin{proof}[Proof of Theorem~\ref{thm-CDC-comparison}]
First, we assume \eqref{eq-sp-comp} and prove \eqref{eq-inclusion}. Let $(\Omega, x_0, R) \in \Gamma_{s_2, p_2}$ and set
\begin{equation*}
K_r \coloneqq \frac{\overline{B_r(x_0)} \setminus \Omega - x_0}{r} \subset \overline{B_1}.
\end{equation*}
By scaling, we have from \eqref{eq:CDC} for $(s_2, p_2)$ that
\begin{equation*}
\mathrm{cap}_{s_2, p_2}(K_r, B_2) \geq c
\end{equation*}
for some $c>0$ independent of $r$. Moreover, it follows from Lemma~\ref{lem-cap-sp-comp} and \cite[Proposition~5.4]{BBK24} that
\begin{equation*}
\mathrm{cap}_{s_1, p_1}(K_r, B_2) \gtrsim C_{s_1, p_1}(K_r) \gtrsim C_{s_2, p_2}(K_r) \gtrsim \mathrm{cap}_{s_2, p_2}(K_r, B_2) \geq c.
\end{equation*}
Scaling back gives
\begin{equation*}
\mathrm{cap}_{s_1, p_1}(\overline{B_r(x_0)} \setminus \Omega, B_{2r}(x_0)) \geq c r^{n-s_1p_1}.
\end{equation*}
Thus, $(\Omega, x_0, R) \in \Gamma_{s_1, p_1}$. This proves $\Gamma_{s_2, p_2} \subset \Gamma_{s_1, p_1}$.

To show that the inclusion in \eqref{eq-inclusion} is strict, we use \cite[Theorem~8.3]{BBK25}, which shows that there exists a compact set $K \subset \mathbb{R}^n$ such that
\begin{equation*}
C_{s_2, p_2}(K) = 0 < C_{s_1, p_1}(K).
\end{equation*}
We may assume that $K \subset B_1 \setminus \overline{B_{1/2}}$ after a translation and dilation if necessary.

We set
\[
\Omega=B_1\setminus\left(\{0\}\cup\bigcup_{i=0}^\infty2^{-i}K\right).
\]
For $0<r<1/2$, choose $i \in \mathbb{N}$ such that $2^{-i} \leq r < 2^{-i+1}$. Using monotonicity and scaling, and then applying \cite[Proposition~5.4]{BBK24}, we obtain
\begin{align*}
\mathrm{cap}_{s_1, p_1}(\overline{B_r} \setminus \Omega, B_{2r})
&\geq \mathrm{cap}_{s_1, p_1}(2^{-i}K, B_{2^{-i+2}}) \\
&=2^{-i(n-s_1p_1)}\mathrm{cap}_{s_1, p_1}(K, B_4) \\
&\gtrsim r^{n-s_1p_1} C_{s_1, p_1}(K),
\end{align*}
which implies that $(\Omega, 0, 1/2) \in \Gamma_{s_1, p_1}$.

On the other hand, it follows from the countable subadditivity of the Sobolev capacity (see \cite[Proposition~3.7]{BBK25}), \cite[Proposition~5.4]{BBK24}, and scaling that
\begin{align*}
C_{s_2, p_2}(\overline{B_r} \setminus \Omega)
&\leq C_{s_2, p_2}(\{0\}) + \sum_{i=0}^\infty C_{s_2, p_2}(2^{-i}K) \\
&\leq C_{s_2, p_2}(\{0\}) + C \sum_{i=0}^{\infty} \mathrm{cap}_{s_2, p_2}(2^{-i}K, B_{2^{-i+2}}) \\
&= C_{s_2, p_2}(\{0\}) + C \sum_{i=0}^{\infty} 2^{-i(n-s_2p_2)} \mathrm{cap}_{s_2, p_2}(K, B_{4}) \\
&\leq C_{s_2, p_2}(\{0\}) + C C_{s_2, p_2}(K).
\end{align*}
By \cite[Lemma~3.4]{BBK25} and $C_{s_2, p_2}(K)=0$, we have $C_{s_2, p_2}(\overline{B_r} \setminus \Omega)=0$. Thus, by \cite[Proposition~5.4]{BBK24} again,
\begin{equation}\label{eq-cap-vanish}
    \mathrm{cap}_{s_2,p_2}(\overline{B_r} \setminus \Omega, B_{2r})=0 \quad \text{for every }0<r\leq 1/2.
\end{equation}
Therefore, $(\Omega, 0, 1/2) \notin \Gamma_{s_2, p_2}$, and \eqref{eq-inclusion} follows.

Next, we assume that \eqref{eq-sp-comp} does not hold. Then either
\begin{equation*}
s_2p_2=s_1p_1\text{ and }p_1>p_2, \quad\text{or} \quad s_2p_2 > s_1p_1.
\end{equation*}
By the preceding argument, there exists $(\Omega , x_0, R) \in \Gamma_{s_2, p_2} \setminus \Gamma_{s_1, p_1}$. Thus, the inclusion \eqref{eq-inclusion} fails to hold. This completes the proof.
\end{proof}

\begin{remark}
The example constructed in the proof of Theorem~\ref{thm-CDC-comparison} has a stronger property. In fact, it follows from \eqref{eq-cap-vanish} that
\begin{equation*}
\int_0^1 \left(\frac{\mathrm{cap}_{s_2, p_2}(\overline{B_{r}} \setminus \Omega, B_{2r})}{r^{n-s_2p_2}} \right)^{\frac{1}{p_2-1}} \frac{\mathrm{d}r}{r} = 0.
\end{equation*}
Therefore, the origin is not merely a failure point for \eqref{eq:CDC} with $(s_2, p_2)$, but also a failure point for the Wiener criterion \eqref{eq-Wiener} with $(s_2, p_2)$.
\end{remark}

\newcommand{\etalchar}[1]{$^{#1}$}


\end{document}